\documentclass[12pt,twoside,reqno]{amsart}
\usepackage[utf8]{inputenc} 
\usepackage[T1]{fontenc}
\usepackage{amssymb}
\usepackage{amsmath}
\usepackage{dsfont}
\usepackage{titlesec}
\titleformat{name=\section}{}{\thetitle.}{0.8em}{\centering\scshape}
\titlespacing*{\section}{0pt}{1.1\baselineskip}{\baselineskip}

\makeatletter
\newcommand\RedeclareMathOperator{%
  \@ifstar{\def\rmo@s{m}\rmo@redeclare}{\def\rmo@s{o}\rmo@redeclare}%
}
\newcommand\rmo@redeclare[2]{%
  \begingroup \escapechar\m@ne\xdef\@gtempa{{\string#1}}\endgroup
  \expandafter\@ifundefined\@gtempa
     {\@latex@error{\noexpand#1undefined}\@ehc}%
     \relax
  \expandafter\rmo@declmathop\rmo@s{#1}{#2}}
\newcommand\rmo@declmathop[3]{%
  \DeclareRobustCommand{#2}{\qopname\newmcodes@#1{#3}}%
}
\@onlypreamble\RedeclareMathOperator
\makeatother

\newtheorem{theorem}{\textbf Theorem}
\newtheorem{lemma}[theorem]{\textbf Lemma}

\theoremstyle{definition}

\theoremstyle{remark}

\usepackage{placeins}

\newcommand{\Z}{\mathbb{Z}}

\newcommand{\F}{\mathbb{F}}
\usepackage{cleveref}
\newcommand{\D}{\mathfrak{D}}
\renewcommand{\O}{\mathcal{O}}
\newcommand{\Q}{\mathbb{Q}}

\DeclareMathOperator{\disc}{disc}

\begin{document}

\title[An Elliptic Curve Analogue of Pillai's Bound]{An Elliptic Curve Analogue of Pillai's Lower Bound on Primitive Roots}

\author[Jin]{Steven Jin}

\address{Department of Mathematics, University of Maryland, College Park MD 20742, USA.}

\email{sjin6816@umd.edu}

\author[Washington]{Lawrence C.  Washington}

\address{Department of Mathematics, University of Maryland, College Park MD 20742, USA.}

\email{lcw@umd.edu}

\subjclass[2020]{11G20, 14H52}

\keywords{elliptic curves, finite fields}

\begin{abstract}
Let $E/\Q$ be an elliptic curve. For a prime $p$ of good reduction, let $r(E,p)$ be the smallest non-negative integer that gives the $x$-coordinate
of a point of maximal order in the group $E(\F_p)$. We prove unconditionally that $r(E,p)> 0.72\log\log p$ for infinitely many $p$, and $r(E,p) > 0.36 \log p$ under the assumption of the Generalized Riemann Hypothesis. These can be viewed as elliptic curve analogues of classical lower bounds on the least primitive root of a prime.
\end{abstract}

\maketitle

\section{Introduction}

Let $E/\F_p$ be an elliptic curve. Recall that there exist unique positive integers $L,M$ such that $$E(\F_p)\cong \Z/L\Z\times \Z/M\Z$$ and $L\mid M$. Here $M$ is the maximal order of a point of $E(\F_p)$.  In order to find a point on $E/\F_p$ of maximal order, a natural strategy is to compute the orders of points with $x$-coordinates  $0, 1, 2, \dots$ and continue until the desired point is found. In practice, this works fairly well. A natural question is how long this process takes in the worst case.

 Along these lines, fix an elliptic curve $E/\Q$  and let $p$ be a prime of good reduction. Let $r(E,p)$ denote the minimal $x$-coordinate of a point of maximal order in
the reduction $E/\F_p$. The goal of this note is to prove the following two lower bounds on $r(E,p)$. 

\begin{theorem} \label{1} 
Let $E/\Q$ be an elliptic curve. There are infinitely many primes $p$ such that $$r(E,p)> 0.72\log\log p.$$
\end{theorem}

\begin{theorem} \label{2}
Let $E/\Q$ be an elliptic curve. Under GRH, there are infinitely many primes $p$ such that $$r(E,p)>0.36\log p.$$
\end{theorem}

These results can be viewed as elliptic curve analogues of lower bounds on the least primitive root $r(p)$ of a prime $p$. Pillai \cite{P1944} proved that 
there is a positive constant $C$ such that $r(p)> C\log \log p$ for infinitely many $p$. Using Linnik's theorem in Pillai's proof, Fridlender \cite{F1949} and Salié \cite{S1949} improved the result to the following. 
We include a proof since it inspired the result of the present paper. 

\begin{theorem}[{{\cite{F1949}, \cite{S1949}}}] \label{3} There exists a positive constant $C$ such that $r(p)> C\log p$ for infinitely many $p$.
\end{theorem}

\begin{proof}  Linnik's theorem says that there exist constants $c$ and $L$ such that every arithmetic progression $a+bn$ with $\gcd(a,b)=1$ contains a prime
$p<c\cdot b^L$. 
For $x>0$ large, let $k=\prod \ell$, where the product is taken over all primes $\ell \le x$. By Linnik's theorem, there exists a prime $p=1$ mod $4k$ with $p<c\cdot (4k)^L$. 
It follows from quadratic reciprocity that every positive prime divisor of $k$ is a quadratic residue for such a prime $p$. This implies that all positive integers $n\leq x$ are quadratic residues for $p$ and therefore cannot be primitive roots mod $p$. By the prime number theorem,  $\log k \sim x$. Therefore, $$r(p) > x\sim \log k> C\log p$$ for some constant $C$ that is independent of $x$. 
\end{proof}

The bound $\log p$ above has been improved to $\log p \log \log \log p$ by Graham-Ringrose \cite{GR} unconditionally and to $\log p \log \log p$ by Montgomery \cite{M} under GRH.

\medskip

For an elliptic curve 
$$
E:\; y^2=f(x)= x^3+Ax+B
$$
with $A, B\in \mathbb  Z$, we follow a similar approach in the search for a prime $p$ for which $r(E,p)$ is large. 
First, we force $E(\F_p)$ to have even order by requiring $f(x)$ to factor into linear factors mod $p$. Next, let $N$ be a large integer. We force all points $(x,y)$ with 
$0\le x\le N$ to be doubles of other points in $E(\F_p)$. Since $E(F_p)$ has even order, these points cannot have maximal order. Finally, we use an explicit version
of the Chebotarev Density Theorem to give an upper bound for $p$ in terms of $N$, which can be transformed into the desired lower bound for $N$ in terms of $p$.

It is reasonable to ask what can be said in terms of an upper bound. In the classical setting, bounding $r(p)$ from above is a well-studied problem (see, for instance, 
 \cite{V1930}, \cite{H1942}, \cite{E1945}, and \cite{ES1957}; see \cite{DD} for computational issues). The best known bound along these lines is the Burgess bound \cite{B1962}, which states that 
$r(p)\ll p^{\frac{1}{4}+\epsilon}$. Under GRH, a result of Shoup \cite{S1992} yields the stronger statement $r(p) \ll \log^6 p$.

Igor Shparlinski has pointed out to us that in the elliptic curve case, one can obtain $r(E,p) = O( p^{\frac{1}{2}+\epsilon})$ via results from \cite{KS}. This is done by using the analogue of the last section of \cite{B1962} and the technique of Theorem 2 of \cite{KS}, combined with the estimate of Theorem 1 of \cite{KS} for characters supported on $\Z/M\Z$. 
Computations suggest that the true order of $r(E,p)$ is smaller and that Theorem \ref{2} is almost sharp, perhaps missing by no more than 
a power of $\log\log p$. See Section \ref{numerical}.

\section{The Proofs}

\noindent \textit{Proof of Theorems \ref{1} and \ref{2}}. Fix $N>0$. Let $E/\Q$ be an elliptic curve over $\Q$ given by Weierstrass equation 
$$y^2=f(x)=x^3+Ax+B,$$
 where $A,B\in \Z$.  We are going to construct a suitable polynomial by constructing several factors and multiplying them together.

Let $\{p_1,\dots ,p_m\}$ consist of the distinct prime divisors of the discriminant of $E/\Q$ and the primes up to and including $7$. Let $g(z)= z^2-\prod_{i=1}^m p_i$.
This polynomial is constructed for the following technical reason: If a prime $p$ is unramified in the splitting field of $g(z)$, then $p$ is a prime of good reduction for $E$.

Recall that the $x$-coordinate of the doubling of a point $(x,y)\ne \infty$ is given by 
$$h(x)=\frac{r(x)}{s(x)}, \: \text{ where } r(x)=x^4-2Ax^2-8Bx+A^2, \:\:  s(x)=4x^3+4Ax+4B.$$
For $0\leq j\leq N$, let $\xi_j(z)=r(z)-js(z)$. 

Let
$$T(z)=f(z)g(z) \prod_{j=0}^N \xi_j(z),$$
and let $F$ be the splitting field of $T(z)$.  If $p$ splits completely in $F/\Q$, then each of the factors of $T(z)$ factors into linear factors over $\F_p$.
Since $p$ is unramified in $F/\Q$ and $g(z)$ is a factor of $T(z)$, it follows that $p>7$ and is a prime of good reduction for $E$.
Since $f(z)$ factors, the 2-torsion is contained in $E(\F_p)$, so $E(\F_p)$ has even order. 

Suppose $P=(j,y)\in E(\F_p)$ for some $0\le j\le N$. Since $\xi_j(z)$ factors into linear factors mod $p$, there exists $x_1\in \F_p$ such that $\xi_j(x_1)\equiv 0\pmod p$.
Let $y_1\in \F_{p^2}$ satisfy $y_1^2\equiv f(x_1)$.
Since the resultant of $r(z)$ and $s(z)$ is $(4A^3+27B^2)^2$, which is not divisible by $p$ by assumption,  we cannot have $s(x_1)\equiv 0\pmod p$. Therefore,
$ (j, y)= 2(x_1, y_1)$ for a suitable choice of sign of $y_1$. Suppose $y_1\not\in \F_p$.  Since $y_1^2\in \F_p$, the Galois conjugate of $y_1$ is $-y_1$. Taking conjugates yields $ (j, y)= 2(x_1, -y_1)= -2(x_1, y_1)= -(j, y)$. Therefore, $(j, y)$ is a point of order 2. Since $p>7$, the Hasse bound implies that $|E(\F_p)| > 4$, so $(j,y)$ cannot be a point of maximal order. On the other hand, if $y_1\in \F_p$, then $(j, y)=2(x_1, y_1)$ implies that $(j, y)$ cannot have maximal order. Therefore, $r(E,p)>N$. 

To finish the proof, we need an upper bound on the smallest $p$. This estimate uses an explicit Chebotarev Density Theorem, which bounds $p$ in terms of the 
discriminant of the splitting field $F$. The next few lemmas bound this discriminant.

\begin{lemma} \label{5}
Let $L/K$ be an extension of number fields with ring of integers $\O_L$ and $\O_K$, respectively. The different $\D_{L/K}$ of $L/K$ is the ideal generated by
$\{g'(\alpha)\}$, where $\alpha$ runs through elements of $\O_L$ such that $L=K(\alpha)$ and $g$ runs through monic polynomials in $\O_K[x]$ satisfying $g(\alpha)=0$. 
\end{lemma}

\begin{proof}
See for instance \cite[Proposition III.3]{LANG}. Note that the usual statement of this result requires $g$ to be the minimal polynomial of $\alpha$. However, if $m(x)$ is the minimal polynomial for $\alpha$, then $g(x)=m(x)h(x)$ for some $h(x)\in \O_K[x]$ and $g'(\alpha)=m'(\alpha)h(\alpha)$, which is in the ideal generated by $m'(\alpha)$. Therefore,  including polynomials $g$ that are potentially reducible does not affect the ideal $\D_{L/K}$. 
\end{proof}

Henceforth, all mentions of discriminants are understood to refer to discriminants over $\Q$. The following result is probably well-known but we could not find a good reference so we include a proof.  

\begin{lemma} \label{4}
Suppose $K/\Q$ is a field extension given as $K=K_1\cdots K_n$. Then $$|\disc(K)|\leq \prod_{i=1}^n |\disc(K_i)|^{[K:K_i]}.$$
\end{lemma}

\begin{proof} (cf. \cite{T1955}) Consider the tower of fields $\Q \subseteq K_1 \subseteq K_1K_2 \subseteq \cdots \subseteq K=K_1K_2\cdots K_n$.
The different of $(K_1K_2\cdots K_{i})/( K_1 K_2 \cdots K_{i-1})$ divides the different $\mathfrak D_i$ of $K_i/\Q$, by Lemma \ref{5}.
Since differents multiply in towers, the different of $K/\Q$ divides $\mathfrak D_1 \mathfrak D_2\cdots \mathfrak D_n$.
Taking the norm from $K$ to $\Q$ yields the result.
\end{proof}

\begin{lemma} \label{6}
Let $f\in \Z[x]$ be a monic polynomial with no repeated roots and let $F$ be the splitting field of $f$. Let $d=[F:\Q]$. Then $\disc(F)^2\text{ divides }\disc(f)^{d}$.
\end{lemma}

\begin{proof}
Let $f(x)=\prod_{i=1}^n (x-\beta_i)$. For $i>0$, let $K_i=\Q(\beta_1,\beta_2,\dots, \beta_i)$. Let $f_i(x)=\prod_{j=i+1}^n (x-\beta_j)$. Since $K_{i+1}=K_{i}(\beta_{i+1})$ and $f_i(\beta_{i+1})=0$, the different $\D_{K_{i+1}/K_i}$ divides  $$f'(\beta_{i+1})=\prod_{j=i+2}^n (\beta_{i+1}-\beta_j)$$ by Lemma \ref{5}. Since differents multiply in 
towers, we have that $\D_{F/\Q}$ divides the ideal generated by $$\prod_{i=0}^{n-2} \prod_{j=i+2}^n (\beta_{i+1}-\beta_j)=\prod_{i<j} (\beta_i-\beta_j)=\disc(f)^{1/2}.$$ Squaring and taking norms, we obtain the result.
\end{proof}

By Lemma \ref{6}, the discriminant of splitting field $K_f$ of $f(z)$  divides $(4A^3+27B^2)^3$. The discriminant of the splitting field $K_g$ of $g(z)$ divides $4\prod_{i=1}^m p_i$. 
A computation shows that the discriminant of $\xi_j(z)$ is 
$$
2^{12}(-4A^3-27B^2)f(j)^2.
$$
Therefore, the discriminant of the splitting field $K_j$ of $\xi_j(z)$ divides $2^{144}(-4A^3-27B^2)^{12}f(j)^{24}$.

\begin{lemma}\label{9}
Let $y^2=x^3+Ax+B$ define an elliptic curve over a field $L$ of characteristic not 2 and assume $E(L)$ contains $E[2]$. Let $j\in L$ and let $\tilde{L}$ be the splitting field
of $$
\xi_j(x) =  x^4-2Ax^2-8Bx+A^2- j(4x^3+4Ax+4B).$$
Then $[\tilde{L} : L]$ divides $4$.
\end{lemma}
\begin{proof} Let $y'=\sqrt{j^3+Aj+B}$ and $L'=L(y')$. Then $(j, y')\in E(L')$. Let $(a, b)\in E(\overline{L})$ satisfy $2(a, b)= (j, y')$ and let $L'(a, b)$
be the field generated by $a$ and $b$.  Since $E[2]\subseteq E(L')$, all four solutions of $2(a, b)= (j, y')$  have coordinates in $L'(a, b)$, and 
$\text{Gal}(L'(a, b)/L')$ is isomorphic to a subgroup of $E[2]$. Since $L'(a, b)$ contains the splitting field of $\xi_j(x)$, we conclude that the splitting
field has degree over $L$ dividing 8.

Note that $-A^3-27B^2$ is a square in $L$, and therefore the discriminant of $\xi_j(z)$ is a square in $L$.
It follows that the Galois group of $\xi_j(x)$ is a subgroup of $A_4$, hence has degree dividing 12. Since the degree also divides 8, the degree divides 4.
\end{proof}

The splitting field $F$ of $T(z)$ is $K_fK_gK_0K_1\cdots K_N$. Therefore,
$$
[F : \Q] \le 6\cdot 2\cdot 4^{N+1}.$$
If $K_j= \Q$ for some $j$, then we can omit that field from our calculations. Therefore, when we apply Lemma \ref {4}  to the present situation,
we can bound the remaining exponents $[F : K_i]$ 
by $[F : \Q]/2$ and obtain
\begin{align*}
|\text{disc}(F)| &\le  \left(C_E \prod_{0\le j\le N}C_E' |f(j)|^{24}\right)^{6\cdot 4^{N+1}}\\
 &\le  \left((C_E'' N)^{72(N+1)}\right)^{6\cdot 4^{N+1}}\\
& = (C_E'' N)^{432(N+1)4^{N+1}},
\end{align*}
where $C_E$, $C_E'$,  and $C_E''$ are constants depending only on the elliptic curve $E$ and where
we have bounded $f(j)$ by a constant times $N^3$.

We now can estimate the smallest prime $p$ that splits completely in $F/\Q$. A theorem of Ahn and Kwon \cite{AK2019} states that 
$p< |\text{disc}(F)|^{12577}$. Therefore,
\begin{align*}
\log\log p & <  \log 12577 + \log\log |\text{disc}(F)|\\
&\le \log 12577+ \log\left(432(N+1)4^{N+1}\right) + \log\log C_E'' N\\
&= N\log 4 + o(N)\\
&<  N/0.72
\end{align*}
when $N$ is sufficiently large. But $r(E,p)> N$ for this $p$, so the proof of Theorem \ref{1} is complete.

Assuming the Generalized Riemann Hypothesis for the Dedekind zeta function of $F$,
 Lagarias and Odlyzko \cite{LO1977} show that there exists $p<C_0(\log(|\text{disc}(F)|))^2$ for some $C_0>0$.  Therefore,
\begin{align*}
\log p  & < \log C_0 + 2\log\log |\text{disc}(F)|\\
&\le 2\log(4) N + o(N)\\
&< N/0.36
\end{align*}
when $N$ is sufficiently large.  Therefore, $r(E,p) > N > 0.36 \log p$ for this $p$. This completes the proof of Theorem \ref{2}.
\qed \smallskip

\noindent \textit{Remark}. As the proof indicates, the constants $0.72$ and $0.36$ can be replaced by any $k_1 <1/\log 4$ and $k_2<1/\log 16$, respectively. The constant $12577$ in the bound of Ahn and Kwon is also not crucial; the existence of such a constant is enough for our purposes. A recent preprint of Kadiri and Wong \cite{KW2021}
improves the constant to $310$.

\section{Numerical Results}\label{numerical}

For each of the elliptic curves in this section, we computed $r(E, p)$ as $p$ ran through primes of good reduction less than $3\times 10^6$ . If a value was
larger than $r(E, q)$ for all $q<p$, we recorded $p$ and $r(E, p)$. The results are given in Tables 1 -- 7.
We omit the data for primes $p<100$ since they are too small to consider in the asymptotic behavior. 
The calculations were done in Sage \cite{sage}.

The third and fourth columns of each table compare $r(E, p)$ to $\log p \log\log p$ and $\log p (\log\log p)^2$. It is well known that 
$\log\log p$ grows so slowly that it is often not easy to recognize what power is appropriate. In the present case,
the ratio of $\log\log(2\times 10^6)$ to $\log\log 200$
is $1.6$, and this is representative of the range of primes in our data. So  the numbers in the third  and fourth columns sometimes exhibit 
 a definite increase or decrease when the power of 
$\log\log p$ is modified. But other times, it is not readily apparent which power is appropriate. 
For each column, we computed the slope of the least-squares line through the data points and listed the result in the last line of the table. For example, for the fourth column of Table 1, we used the points 
$(1, 1.49), (2, 0.94), (3, 1.10), \dots, (14, 1.38)$. The least-squares line has slope $.018$.  In three of the tables, the absolute value of the slope is smaller in the third column
and in the other four tables the absolute value of the slope is smaller in the second column.
We do not have an explanation for the potential variation
of exponents. It seems reasonable to guess that  an upper bound of the form $r(E, p) \le C \log p (\log\log p)^{\delta}$ is possible.
In other words, the estimate of Theorem 2 is probably sharp, except for powers of $\log\log p$ and smaller contributions. As mentioned in the Introduction, Montgomery \cite{M} showed under GRH that the smallest quadratic non-residue is $\Omega(\log p \log\log p)$. The
numerical results for elliptic curves indicate that a similar result is possible for elliptic curves.(Of course the estimate of Theorem 1 is probably not close to sharp, unless GRH is false.)

The first four curves have complex multiplication by  $\mathbb Z[i]$, ,$\mathbb Z[i]$, $\mathbb Z[(1+\sqrt{-3})/2]$, and $\mathbb Z[(1+\sqrt{-7})/2]$ , respectively.
All of the $p$ that occur are supersingular primes for their respective  curves with the exception of $p=13007$ in Table 4.  For these supersingular primes, 
the group $E(\F_p)$ has order $p+1$ and is either cyclic or cyclic times a group of order 2. This can be seen as follows.
The Frobenius map is given by $\sqrt{-p}$ in the endomorphism ring. If the full $n$-torsion is contained in $E(\F_p)$, then
the Frobenius endomorphism must be congruent to 1 mod $n$. But $(\sqrt{-p}-1)/n$ is not integral when $n>2$. It follows that $E(\F_p)$ is either $\Z/(p+1)/Z$, or $\Z/\frac{p+1}{2}\Z\times \Z/2\Z$. The latter is always the case for the curve $y^2=x^3-x$. However,  
for the other three curves, only one point of order 2 is in $E(\F_p)$, so the group is cyclic. 
A cyclic group sometimes has fewer elements of maximal order than a non-cyclic abelian group of the same order. However, it is not 
clear why almost every example is a supersingular prime.

The curves in the last three tables do not have complex multiplication (the last curve is the Weierstrass form for $X_0(11)$).
The values of $r(E, p)$ are somewhat smaller than those for the curves with complex multiplication. Perhaps this reflects the fact that supersingular primes are less frequent, but a good explanation is yet to be found. 

The curves in Tables 2 and 7 have non-cyclic 2-torsion over $\mathbb Q$, hence mod each of the primes $p$ considered. This phenomenon 
seems to cause larger values of $r(E,p)$.

An interesting situation occurs in Tables 1 and 2, where the prime 537599 is in both tables. Note that in Table 1, the group
$E(\F_p)$ is cyclic of order $537600 = 2^{10}\cdot 3\cdot 5^2\cdot 7$, which is a very smooth number. This lowers the probability
that a randomly chosen element is a generator. In fact, $\phi(537600)/537600= 8/35$. The group for the curve in Table 2 is the product of a 
cyclic group of order $2^{9}\cdot 3\cdot 5^2\cdot 7$ times a group of order 2. The probability is again 8/35 that a randomly chosen element of the group has maximal order.  These probabilities are low, but it still seems to be a lucky coincidence that this $p$ occurs in both tables. The smoothness of the group order is probably not the deciding factor.
There are several smooth numbers close to each of the primes in our table. For example, the Mersenne prime $2^{19}-1= 524287$ yields $r(E,p)=3$ for $y^2=x^3+x$ and $r(E,p)= 4$ for $y^2=x^3-x$, and both curves have $2^{19}$ points. The more relevant property might be the existence of several small prime factors of $n=p+1$ (in the supersingular case for the present curves) since this makes $\phi(n)/n$ small. But this does not guarantee that $r(E,p)$ is large. An example is $p=570569$, where $p+1=2\cdot 3\cdot 5\cdot 7\cdot 11\cdot 13\cdot 19$. But $r(E,p)=6$ for both $y^2=x^3+x$ and $y^2=x^3-x$. It would be interesting to find a good explanation, if one exists, for the double occurrence of 537599.

\begin{table}[h!]
\begin{center}
\caption{$y^2=x^3+x$}
\label{table2}
\begin{tabular}{|c|c|c|c|}\hline
\boldmath{$p$} & \boldmath{$r(E, p)$}  &\boldmath{$r(E, p)/ \log p \log\log p $} & \boldmath{$r(E, p)/ \log p(\log\log p)^2 $}\\
\hline
179 & 21 &2.46& 1.49\\
\hline
719 & 22 & 1.78 & 0.94\\
\hline
743 & 26 & 2.08 & 1.10\\
\hline
1559 & 31 & 2.11 & 1.06\\
\hline
1931 & 47 & 3.07  &1.52\\
\hline
5039 & 51 & 2.79 &1.30\\
\hline
9239 & 58 & 2.87 &1.30\\
\hline
23399 & 62 & 2.67 &1.16\\
\hline
23663 & 79 & 3.40 &1.47\\
\hline
52919 & 109 & 4.20 &1.76\\
\hline
407879 & 114 & 3.45 &1.35\\
\hline
537599 & 116 & 3.41 & 1.32\\
\hline
2599559 & 139 & 3.49 & 1.30\\
\hline
2611391 & 148 & 3.72 & 1.38\\
\hline
slope & & $.140$ & $.018$\\
\hline
\end{tabular}
\end{center}
\end{table}

\FloatBarrier

\begin{table}[h!]
\begin{center}
\caption{$y^2=x^3-x$}
\label{table3}
\begin{tabular}{|c|c|c|c|}\hline
\boldmath{$p$} & \boldmath{$r(E, p)$}  &\boldmath{$r(E, p)/ \log p \log\log p $}& \boldmath{$r(E, p)/ \log p(\log\log p)^2 $}\\
\hline
191 & 20 &2.30& 1.38\\
\hline
311 & 22 &2.19 & 1.26\\
\hline
431 & 27 &2.47 & 1.37\\
\hline
479 & 37 &3.29 & 1.81\\
\hline
1319 & 38 &2.68 & 1.36\\
\hline
2351 & 40 &2.51 & 1.23\\
\hline
3119 & 60 &3.58 & 1.72\\
\hline
5711 & 61 &3.27 & 1.51\\
\hline
7559 & 67 &3.43 & 1.57\\
\hline
13679 & 84 &3.91 & 1.74\\
\hline
26759 & 86 &3.63 & 1.56\\
\hline
49871 & 102 & 3.96 & 1.66\\
\hline
115079 & 123 & 4.30 & 1.75\\
\hline
327599 & 130 &4.03& 1.58\\
\hline
340031 & 133 &4.10 & 1.61\\
\hline
504479 & 157 &4.64 & 1.80\\
\hline
537599 & 192 &5.64 & 2.19\\
\hline
slope & & $.169$ & $.033$\\
\hline
\end{tabular}
\end{center}
\end{table}

\begin{table}[h!]
%\begin{center}
\caption{$y^2=x^3+1$}
\label{table1}
\begin{tabular}{|c|c|c|c|}\hline
\boldmath{$p$} & \boldmath{$r(E, p)$} & \boldmath{$r(E, p)/ \log p\log\log p$} & \boldmath{$r(E, p)/ \log p(\log\log p)^2$}\\
\hline
101 & 28 & 3.97 & 2.59\\
\hline
479 & 40 & 3.56 & 1.96\\
\hline
569& 45 & 3.84 & 2.08 \\
\hline
1319 & 46 & 3.25 & 1.65 \\
\hline
2999 & 67 & 4.02 & 1.93 \\
\hline
38639 & 105 & 4.22 & 1.79 \\
\hline
149519 & 112 & 3.79 & 1.53 \\
\hline
403079& 114 & 3.45 & 1.35\\
\hline
1385039 & 116 & 3.10 & 1.17\\
\hline
2837519 & 144 & 3.59& 1.33\\
\hline
slope & & $-.041$ & $-.127$\\
\hline
\end{tabular}
%\end{center}
\end{table}

\begin{table}[h!]
\begin{center}
\caption{$y^2=x^3-385875x-113447250$}
\label{table4}
\begin{tabular}{|c|c|c|c|}\hline
\boldmath{$p$} & \boldmath{$r(E, p)$} & \boldmath{$r(E, p)/ \log p \log\log p $} & \boldmath{$r(E, p)/ \log p(\log\log p)^2 $}\\
\hline
167 & 13 & 1.56 & 0.95\\
\hline
241 & 25 & 2.68 & 1.57\\
\hline
593 & 27 & 2.28 &  1.23\\
\hline
2063 & 31 & 2.00 & 0.98\\
\hline
3527 & 38 & 2.22 & 1.05 \\
\hline
9203 & 40 & 1.98 & 0.90\\
\hline
13007 & 42 & 1.97 & 0.88\\
\hline
13859 & 59 & 2.74 & 1.22\\
\hline
174569 & 70 & 2.33 & 0.93 \\
\hline
2798459 & 78 & 1.95 & 0.72 \\
\hline
slope & & $.018$ & $-.043$\\
\hline
\end{tabular}
\end{center}
\end{table}

\begin{table}[h!]
\begin{center}
\caption{$y^2=x^3+x+1$}
\label{table5}
\begin{tabular}{|c|c|c|c|}\hline
\boldmath{$p$} & \boldmath{$r(E, p)$} &\boldmath{$r(E, p)/ \log p \log\log p $} &  \boldmath{$r(E, p)/ \log p(\log\log p)^2 $}\\
\hline
197 & 8 & 0.91 & 0.55\\
\hline
283 & 13 &1.33 & 0.77\\
\hline
613 & 17 & 1.42 & 0.77\\
\hline
647 & 18 & 1.49 & 0.80\\
\hline
811 & 19 &1.49 & 0.78\\
\hline
1187 & 29 &2.09 & 1.07\\
\hline
21023 & 31 &1.36 & 0.59\\
\hline
29669 & 32 & 1.33 & 0.57\\
\hline
60317 & 42 & 1.59 &0.66\\
\hline
76421 & 48 &1.76  &0.73\\
\hline
114269 & 51 & 1.78 & 0.73\\
\hline
250993 & 60 & 1.91& 0.76\\
\hline
2800267 & 64 &1.60 & 0.59\\
\hline
slope & & $.048$ & $-.005$\\
\hline
\end{tabular}
\end{center}
\end{table}

\begin{table}[h!]
\begin{center}
\caption{$y^2=x^3-13392x - 1080432$}
\label{table6}
\begin{tabular}{|c|c|c|c|}\hline
\boldmath{$p$} & \boldmath{$r(E, p)$} & \boldmath{$r(E, p)/ \log p \log\log p $}& \boldmath{$r(E, p)/ \log p (\log\log p)^2 $}\\
\hline
107 & 18 & 2.50 &1.62 \\
\hline
227 & 25 & 2.73& 1.61  \\
\hline
461 & 28 & 2.52& 1.39 \\
\hline
997 & 30 & 2.15& 1.16 \\
\hline
3613 & 37 & 2.15& 1.02 \\
\hline
20173 & 49 & 2.16& 0.94 \\
\hline
77813 & 51 & 1.87& 0.77 \\
\hline
93419 & 64 & 2.29& 0.94 \\
\hline
508213 & 81 & 2.39& 0.93 \\
\hline
2311823& 96 & 2.44 & 0.91\\
\hline
slope & & $-.030$ & $-.089$\\
\hline
\end{tabular}
\end{center}
\end{table}

\begin{table}[h]
\begin{center}
\caption{$y^2=x^3-7x +6$}
\label{table7}
\begin{tabular}{|c|c|c|c|}\hline
\boldmath{$p$} & \boldmath{$r(E, p)$} & \boldmath{$r(E, p)/ \log p \log\log p $}& \boldmath{$r(E, p)/ \log p (\log\log p)^2 $}\\
\hline
101 & 28 & 3.97 & 2.59\\
\hline
1297 & 30& 2.12 & 1.08\\
\hline
1511 & 34 & 2.33 & 1.17\\
\hline
1873 & 56& 3.68 & 1.82 \\
\hline
12119 & 68 & 3.23 & 1.44\\
\hline
12239 & 71 & 3.36 & 1.50\\
\hline
41039 & 74 & 2.95 & 1.25\\
\hline
47351 & 75 & 2.93 & 1.23\\
\hline
64679 & 91 & 3.42 & 1.42\\
\hline
178559 & 110 & 3.65 & 1.46\\
\hline
393121 & 142 & 4.31 & 1.69\\
\hline
1161599 & 169 & 4.59 & 1.74\\
\hline
2671679 & 194 & 4.87 & 1.81\\
\hline
slope & & $.140$ & $-.004$ \\
\hline
\end{tabular}
\end{center}
\end{table}

\FloatBarrier


\begin{thebibliography}{99}

\bibitem{AK2019} J.-H. Ahn and S.-H. Kwon, An explicit upper bound for the least prime ideal in the Chebotarev density theorem. \textit{Annales de l'Institut Fourier}, 69(3): 1411--1458, 2019.

\bibitem{B1962} D. A. Burgess, On character sums and primitive roots, \textit{Proc. London Math. Soc. (3)}, 12: 179--192, 1962.

\bibitem{DD} J. Dubrois and J-G. Dumas,  Efficient polynomial time algorithms computing industrial-strength primitive roots,
  \textit{Inform. Process. Lett.},   97(2): 41--45, 2006.


\bibitem{E1945} P. Erdős, Least primitive root of a prime, \textit{Bull. Amer. Math. Soc.}, 55: 131--132, 1945.

\bibitem{ES1957} P. Erdős and H. N. Shapiro, \textit{Pacific J. Math.}, 7(1): 861--865, 1957.

\bibitem{F1949} V. R. Fridlender, \textit{Proc. USSR Acad. Sci.}, 66: 351--352, 1949.

\bibitem{GR} S. W. Graham, C. J. Ringrose, Lower bounds for least quadratic non-residues. In \textit{Analytic Number Theory (Allerton Park, IL, 1989)}, Progr. Math., 85, Birkhauser, Boston, MA, 269--309, 1990.

\bibitem{H1942} L. K. Hua, On the least primitive root of a prime, \textit{Bull. Amer. Math. Soc.}, 48: 726--730, 1942.

\bibitem{KW2021} H. Kadiri and P. Wong, Primes in the Chebotarev density theorem for all number fields (with an appendix by Andrew Fiori), arxiv.org/pdf/2105.14181.pdf

\bibitem{KS} D. R. Kohel and I. E. Shparlinski, On Exponential Sums and Group Generators for Elliptic Curves over Finite Fields. In \textit{Algorithmic Number Theory}. Lecture Notes in Computer Science, vol 1838. Springer, Berlin, Heidelberg, 2000.

\bibitem{M} H. L. Montgomery, \textit{Topics in Multiplicative Number Theory}, Lecture Notes in Math. 227, Springer-Verlag, New York, 1971.

\bibitem{LO1977}  J. C. Lagarias and A. M. Odlyzko, Effective versions of the Chebotarev density theorem. In \textit{Algebraic number fields: L-functions and Galois properties (Proc. Sympos., Univ. Durham, Durham, 1975)}, 409--464. Academic Press, London, 1977.

\bibitem{LANG} S. Lang, \textit{Algebraic number theory}, Graduate Texts in Math. 110, Springer-Verlag, New York, 1986.

\bibitem{P1944} S. Pillai, On the smallest primitive root of a prime, \textit{J. Indian Math. Soc.}, 8: 14--17, 1944.


\bibitem{S1949} H. Salié, Über den kleinsten positiven quadratischen Nichtrest nach einer Primzahl, \textit{Math. Nachr.}, 3: 7--8, 1949.

\bibitem{S1992} V. Shoup, Searching for primitive roots in prime fields, \textit{Math. Comp.}, 58: 369--380, 1992.

\newcommand{\etalchar}[1]{$^{#1}$}
\bibitem{sage}
W.\thinspace{}A. Stein et~al., \emph{{S}age {M}athematics {S}oftware (Online version: April 2021)} The Sage Development Team, {\tt https://sagecell.sagemath.org}.

\bibitem{T1955} H. Toyama, A note on the different of the composed field. \textit{Kodai Math. Sem. Rep.}, 7(2): 43--44, 1955.

\bibitem{V1930} I. M. Vinogradov, On the least primitive root of a prime, \textit{Dokl. Akad. Nauk, S.S.S.R.}, 7--11, 1930. 

\end{thebibliography}
\end{document}